\documentclass[a4paper, reqno]{amsart}
\usepackage{a4,wasysym}
\usepackage{amsmath,graphicx, amssymb,color,ulem,bbm}

\newtheorem{theorem}{Theorem}

\newtheorem{proposition}[theorem]{Proposition}

\newcommand{\ud}{\mathrm{d}}

\begin{document}

\title[Boundary perturbations and steady states of structured populations]{Boundary perturbations and steady states of structured populations}

\author[\`{A}. Calsina]{\`{A}ngel Calsina}
\address{\`{A}ngel Calsina, Department of Mathematics, Universitat Aut\`{o}noma de Barcelona, Bellaterra, 08193, Spain}
\email{acalsina@mat.uab.es}

\author[J. Z. Farkas]{J\'{o}zsef Z. Farkas}
\address{J\'{o}zsef Z. Farkas, Division of Computing Science and Mathematics, University of Stirling, Stirling, FK9 4LA, United Kingdom }
\email{jozsef.farkas@stir.ac.uk}

\subjclass{92D25, 35L04, 34K30}
\keywords{Structured populations, boundary perturbation, spectral theory of positive operators.}
\date{\today}

\begin{abstract}

In this work we establish conditions which guarantee the existence of (strictly) positive steady states of a nonlinear structured population model. In our framework the steady state formulation amounts to recasting the nonlinear problem as a family of eigenvalue problems combined with a fixed point problem. Amongst other things our formulation requires us to control the growth behaviour of the spectral bound of a family
    of linear operators along positive rays. For the specific class of model we consider here, this presents a considerable
     challenge. We are going to show that the spectral bound of the family
       of operators, arising from the steady state formulation, can be controlled by perturbations in the domain of the
        generators (only). These new boundary perturbation results are particularly important for models exhibiting fertility controlled dynamics. As an important by-product of the application of the  boundary perturbation results we employ here we recover (using
          a recent theorem by H.\,R. Thieme) the familiar net reproduction number (or function) for models with single state at
           birth, which include for example the classic McKendrick (linear) and Gurtin-McCamy (non-linear) age-structured
            models.

\end{abstract}
\maketitle

\section{Recap and motivation}
In recent years we have developed a general framework to treat steady state problems of nonlinear partial differential and partial integro-differential equations, see \cite{CF,CF2,CF3}. We have been mainly interested in studying models of physiologically structured populations (for reference on structured population models see for example the monographs \cite{CUS,Iannelli,MD,W}). The types of nonlinearities naturally arising in these models present considerable challenges. For example a basic size-structured population model can be formulated as a quasi-linear first order hyperbolic partial differential equation with a nonlinear and non-local boundary condition. On the other hand, the method we have developed to treat the steady state problem capitalises on the fact that the models we have considered describe the dynamics of populations, and consequently in our framework we have to study families of positive linear operators, admitting desirable spectral properties. In particular, the operators that na\-tu\-rally arise generate positive strongly continuous semigroups, often with further desirable re\-gu\-larity properties, such as eventual compactness and irreducibility. The spectral theory of these operators is well developed (at least on abstract Lebesgue spaces), see for example the monographs \cite{AGG,Clement,NAG,Sch}.

To start let us briefly recall from \cite{CF2} how we can (re)formulate the positive steady state problem for a nonlinear evolution equation, in general. We are then going to show in the subsequent section how to apply this general framework to a specific class of model exhibiting fertility controlled dynamics. Let $\mathcal{X}$ and $\mathcal{Y}$ be Banach lattices, and let us consider the (parametrised) family of abstract Cauchy problems.
\begin{equation}\label{problem}
\frac{d u}{d t}=\mathcal{A}_{\bf u}\,u,\quad D(\mathcal{A}_{\bf u})\subseteq\mathcal{X},\quad \quad u(0)=u_0.
\end{equation}
Above, for every ${\bf u}\in \mathcal{Y}$, $\mathcal{A}_{\bf u}$ is a linear operator, typically, the generator of a strongly continuous semigroup on the state space $\mathcal{X}$. For many of the concrete applications we have considered (see e.g. \cite{CF,CF2,FGH}) $\mathcal{X}$ can be naturally chosen as the Lebesgue space of integrable functions; while in the simplest case (as for example in case of the model we have considered in \cite{FGH}) we have $\mathcal{Y}=\mathbb{R}$.

The relationship between ${\bf u}\in\mathcal{Y}$-parameter space, and $u\in\mathcal{X}$-state
space, is determined in general by the so-called environmental operator:
\begin{equation}\label{envcond}
E\,:\,\mathcal{X}\to\mathcal{Y},\quad E(u)={\bf u}.
\end{equation}
The latter condition we call the environmental condition, or environmental feedback. From the biological point of view, $E$ determines how the standing population affects (via population level interaction) individual development, hence the terminology (see e.g. \cite{DGKMT,DGM} for more details). From the mathematical point of view, the parametrisation can often be chosen in such a way that $E$ is a positive linear operator, for example an integral operator. In particular, if problem \eqref{problem}-\eqref{envcond} can be set up such that the range of $E$ is contained in $\mathbb{R}^n$ for some $n\in\mathbb{N}$, then we say that problem \eqref{problem} incorporates $n$-dimensional nonlinearity. Otherwise we say that our model incorporates an infinite dimensional nonlinearity. We note that for models with infinite dimensional nonlinearities (e.g. the one we have considered in \cite{CF}) it is often desirable (from the mathematical point of view) to set $\mathcal{Y}=\mathcal{X}$.
Also note that (in particular also in the case when $\mathcal{Y}=\mathcal{X}$), $E$ is not necessarily surjective.

The advantage of the abstract formulation above is that the (positive) steady state problem can be formally written as:
\begin{equation}\label{problemss}
\mathcal{A}_{\bf u}\, u=0\,u,\quad \Psi(E(u))={\bf u};\quad 0\not\equiv u\in\mathcal{X}_+,
\end{equation}
where $\Psi:\mathcal{Y}_+\to\mathcal{Y}_+$ is an appropriately
defined non-linear map. At the same time, (at least) in case of (structured)
population models, the  generators $\mathcal{A}_{\bf u}$ naturally
determine well-posed positive linear evolution problems. There is a
well-developed mathematical framework to analyse these linear
evolution problems, for example using the theory of positive
strongly continuous semigroups.  Solving problem \eqref{problemss}
(that is, establishing the existence of a positive steady state of
our nonlinear model) then amounts to study the spectral behaviour of
the family of operators $\{\mathcal{A}_{\bf u}\,|\,{\bf u}\in
\mathcal{Y}\}$ (note that $\mathcal{Y}$ is an infinite dimensional
vector space, in general), and to prove the existence of a fixed
point of a (possibly set-valued and non-convex) nonlinear map, which
is related to the feedback condition $\Psi(E(u))={\bf u}$ in
\eqref{problemss}. Both of these problems pose different types of
challenges depending on the type of nonlinearity, and the structure
of the model equations, we are dealing with. In the present paper we
are going to tackle one of these important challenges through  a concrete
example, and therefore further develop, and at the same time also
demonstrate, the applicability of the general theory outlined in \cite{CF2}.

It is evident that there are natural models of biological populations which give rise to equations incorporating different types of nonlinearities, in particular both finite and infinite dimensional ones, simultaneously. A case in point is a population exhibiting a hierarchical structure, in which the hierarchy affects individual mortality.  More specifically we have a population in mind in which older (more experienced) individuals possess an advantage when competing for resources, for example food. At the same time we assume that individual fertility rate is modulated by the total population size, due to scramble competition effects.

As a motivating example consider the following age-specific mortality rate $\mu$, and fertility rate $\beta$.
\begin{equation}\label{vitalex}
\mu\left(a,\frac{\int_a^m u(r,t)\,\ud r}{\int_0^m u(r,t)\,\ud r}\right),\quad \beta\left(a,\int_0^m u(r,t)\,\ud r\right).
\end{equation}
Above $a\in [0,m]$ denotes age, while $u(a,t)$ denotes the density of individuals of age $a$ at time $t$. Note that while the mortality of an individual of age $a$ depends on the proportion of older individuals in the population, the fertility of an individual of age $a$ depends on the total number of individuals in the population. In particular, the corresponding age-structured partial differential equation model (see the next section) describes the time evolution of a hierarchic age-structured population, where individuals in the population compete with older individuals for resources, e.g. food, affecting their mortality. Similar hierarchic age-structured population models were in fact already introduced and studied for example by Cushing in \cite{Cushing}. Importantly, here  we assume that individual mortality is affected by the proportion of older individuals, while individual fertility (or the success of mating) is affected by the total population size. This in fact can be viewed as mortality is affected by pure contest competition (mass independent), while fertility is affected by scramble (mass dependent) competition.
Realistic ecological situations may include populations exhibiting cannibalistic behaviour, where the death rate of individuals is modulated by the proportion of older (larger) individuals in the population. 
In general, the effects of frequency dependent survival probabilities have been studied in the case of specific real world (e.g. fish) populations, see e.g. \cite{Olendorf}.

It turns out however that incorporating certain types of nonlinearities (such as the ones above in \eqref{vitalex})  in a
 structured population model, poses some challenges, when trying to apply the abstract framework \eqref{problem}-\eqref{problemss} 
 to study the existence of positive steady states. In particular, our steady state problem formulation in \eqref{problemss} requires, amongst other things, to control the spectral bound $s(\mathcal{A}_{\bf u})$ on positive rays in the parameter space $\mathcal{Y}$. But as we can see $\mu$ is in fact constant along such positive rays, since clearly
\begin{equation*}
\mu\left(\cdot,\frac{\int_\cdot^m u(r,t)\,\ud r}{\int_0^m u(r,t)\,\ud r}\right)\equiv\mu\left(\cdot,\frac{\int_\cdot^m \alpha\,u(r,t)\,\ud r}{\int_0^m \alpha\,u(r,t)\,\ud r}\right), \quad \forall\,\alpha>0.
\end{equation*}

Hence the information how the spectral bound of the generator $\mathcal{A}_{\bf u}$ changes with respect to perturbations of ${\bf u}$ along rays is encoded in the fertility function $\beta$. In general for distributed states at birth models (such as the one we considered in \cite{FGH}), since the recruitment operator appears as a bounded perturbation of the main part of the generator, this does not pose additional problems. However, for single state at birth models, such as the one we consider here, the function $\beta$ only plays a role in the definition of the domain of the unbounded generator $\mathcal{A}_{\bf u}$, and as a consequence any useful information pertaining the growth (or decay) of its spectral bound can only be extracted from its domain $D(\mathcal{A}_{\bf u})$. Our main aim in this work is to show (in the next section) that we can deal with such domain perturbation problems by using the framework of Sobolev towers, as we have already indicated in \cite{CDF}.

As a by-product, using the Sobolev tower construction and the boundary perturbation results contained in \cite{NAG}, together with a  recent result of H.\,R. Thieme from \cite{HT}, we recover the familiar net reproduction number/function. In particular, we rigorously justify that the biologically plausible definition of the net reproduction number (see e.g. \cite{Iannelli}) - for the linear McKendrick model -, and  function (see e.g. \cite{FH}) - for the nonlinear Gurtin-MacCamy age-structured model -, is indeed the spectral radius of an appropriate operator, which arises when we apply the natural splitting to the generator $\mathcal{A}_{\bf u}$. Note that the idea (which we proposed earlier in \cite{CDF}, see in particular the Concluding Remarks section) to lift the age-structured problem into the extrapolated space is crucial, as this allows us to apply the necessary and natural splitting of the generator $\mathcal{A}_{\bf u}$ which is required to apply Thieme's theorem and to arrive at the biologically plausible net reproduction number/function (see also \cite{Farkas2018} for more details in the direction of net reproduction functions).

\section{Boundary perturbations - through a concrete example}

Our general goal in this section is to show how the spectral bound of the family of operators $\{\mathcal{A}_{\bf u}\,|\,{\bf u}\in\mathcal{Y}_+\}$ can be controlled (along positive rays in $\mathcal{Y}$) by perturbations in the domain of $\mathcal{A}_{\bf u}$ only. We are going to demonstrate this through a concrete example, but the approach we present here is quite general and can be applied to a variety of models with non-local boundary conditions (single state at birth).  Here as an example we consider a model of a hierarchic age-structured population with mortality and fertility rates as described in \eqref{vitalex}. In particular, with the specific choice of the vital rates in \eqref{vitalex} the familiar (single state at birth) Gurtin-McCamy-type age-structured population model reads as follows.
\begin{align}
p_t(a,t)+p_a(a,t)= & -\mu\left(a,\frac{\int_a^m p(r,t)\,\ud r}{\int_0^m p(r,t)\,\ud r}\right)p(a,t),\quad a\in (0,m), \quad m<\infty, \label{heq1} \\
p(0,t)=& \int_0^m \beta\left(a,\int_0^m p(r,t)\,\ud r\right)p(a,t)\,\ud a, \quad t>0,\quad p_0(a):=p(a,0). \label{heq2}
\end{align}

We set the state space as $\mathcal{X}=L^1(0,m)$, and we note that our model \eqref{heq1}-\eqref{heq2} is highly nonlinear.
In particular the nonlinearities are induced by two essentially different type of competition effects (scramble and contest), which are incorporated in the model via the two operators $F_1$ and  $F_2$, which are defined as follows.
\begin{equation}
F_1(p)(\cdot)=\frac{\int_\cdot^m p(r)\,\ud r}{\int_0^m p(r)\,\ud
r},\quad p\in\mathcal{X}_+\setminus\{ 0\},\quad F_2(p)=\int_0^m
p(r)\,\ud r,\quad p\in\mathcal{X}_+. \label{heq3}
\end{equation}
The results obtained for this model will therefore also illustrate that the general framework we have outlined in \cite{CF2} can be applied to models incorporating multiple (different) types of competition effects (inducing different types of nonlinearities), simultaneously.

To formulate the (positive) steady state problem for  model \eqref{heq1}-\eqref{heq2} in the form of \eqref{problemss}, we set the parameter space as $\mathcal{Y}:=L^1(0,m)=\mathcal{X}$, and we define for every ${\bf u}\in\mathcal{Y}_+^0:=\mathcal{Y}_+\setminus\{{\bf 0}\}$ the linear operators $\mathcal{B}_{\bf u}$ and $\Phi_{\bf u}$ as follows.
\begin{align}
\mathcal{B}_{\bf u}\,p= &-p'-\mu\left(\cdot,F_1({\bf u})(\cdot)\right)p, \label{heq4} \\
D(\mathcal{B}_{\bf u})= & \left\{p\in W^{1,1}(0,m)\,|\, p(0)=\Phi_{\bf u}(p)\right\}, \label{heq5} \\
\Phi_{\bf u}\,p=& \int_0^m\beta\left(a,F_2({\bf u})\right)p(a)\,\ud
a, \quad D(\Phi_{\bf u})=\mathcal{X}_+. \label{heq6}
\end{align}

In \eqref{heq5} above $W^{1,1}(0,m)$ denotes the usual Sobolev space of integrable functions with generalized first derivative in $L^1$. For every ${\bf u}\in\mathcal{Y}_+^0$ it can be shown that $\mathcal{B}_{\bf u}$
generates a positive and eventually compact semigroup of operators, under
some very mild smoothness assumptions on the model ingredients $\beta$ and $\mu$ (see e.g. \cite{FH} for more details). For example, we may, and we will, in the rest of the paper assume that
$\mu$ and $\beta$ are non-negative, continuous and bounded, and that $\mu\not\equiv 0,\beta\not\equiv 0$.

Moreover, the semigroup generated by $\mathcal{B}_{\bf u}$ can be shown irreducible if we impose the natural assumption that
\begin{equation}\label{irred}
\beta(m,\cdot)\ne 0.
\end{equation}

To establish the existence of a positive steady state of model \eqref{heq1}-\eqref{heq2}, using the general framework from  \cite{CF2}, we need to assure (amongst other things) the existence of a `desirable' set in the parameter space $\mathcal{Y}$, where the spectral bound of the operators $\mathcal{B}_{\bf u}$ vanish.
In particular, our goal is to show that this level set, defined specifically as
\begin{equation}\label{levelset}
{S}=\{{\bf u}\in\mathcal{Y}^0_+\,|\, s(\mathcal{B}_{\bf u})=0\},
\end{equation}
intersects every positive ray in $\mathcal{Y}$ in a unique point (element), and that it is contained in a ball $B\subset\mathcal{Y}$ of a sufficiently large radius. To this end, it is sufficient to show that the spectral bound is strictly monotone decreasing (or, alternatively, increasing) along positive rays, that is $\forall\,{\bf u}\in\mathcal{Y}^0_+,$ we have $s\left(\mathcal{B}_{\bf u}\right)>s\left(\mathcal{B}_{\alpha\,{\bf u}}\right),\, \forall\,\alpha>1$; and in particular that it changes sign along every ray in the set $B\cap\mathcal{Y}_+$. But, as we noted before, $\mu$ is constant along positive rays (since $F_1$ is constant along rays in $\mathcal{Y}$), hence to establish that the spectral bound is strictly monotone along rays in the positive cone we need to control the spectral bound via the boundary operator $\Phi_{\bf u}$, which only appears in the domain of the generator $\mathcal{B}_{\bf u}$.

It is the main purpose of this paper to show how this can be done using the framework of Sobolev towers, in particular using the extrapolated space of $\mathcal{X}$. To the best of our knowledge the notion of  extrapolated  spaces was originally introduced to establish existence of solutions of some nonlinear evolution equations, see e.g. \cite{PG}. Significantly, later they were also employed to deal with boundary perturbations of generators of semigroups, see e.g. \cite{D-Sch,Gre} for more details. We also mention that there is a well-developed and fruitful theory of integrated semigroups, in particular with applications to age-structured population models. We refer the interested reader for example to \cite{Magal2018,Magal2009} for more details and further references.

Recall from \cite[Ch.II]{NAG} that the extrapolated space $\mathcal{X}_{-1}$ corresponding to the generator $\mathcal{A}$ (assuming it is invertible) is
defined as the completion of $\mathcal{X}$ with respect to the following norm:
\begin{equation}
||x||_{-1}:=\left|\left|\mathcal{A}^{-1}\,x\right|\right|_{\mathcal{X}}.
\end{equation}
Then, $\mathcal{X}$ is dense in $\mathcal{X}_{-1}$ and the generator $\mathcal{A}$ of the strongly continuous semigroup $\mathcal{T}$ has a unique continuous extension $\mathcal{A}_{-1}$ with $D(\mathcal{A}_{-1})=\mathcal{X}$, which is the generator of the so-called extended semigroup $\mathcal{T}_{-1}$ on $\mathcal{X}_{-1}$. Recall also from \cite[Ch.II]{NAG} that in general $\mathcal{A}_|$ is called the part of $\mathcal{A}$ in $\mathcal{Z}\subseteq\mathcal{X}$ if
\begin{equation}
\mathcal{A}_|\,z:=\mathcal{A}\,z,\quad D(\mathcal{A}_|):=\left\{z\in D(\mathcal{A})\cap\mathcal{Z}\,|\,\mathcal{A}\,z\in\mathcal{Z}\right\}.
\end{equation}
Note that $\mathcal{A}$ is the part of $\mathcal{A}_{-1}$ in $\mathcal{X}$.

In our setting we consider the modified generator
\begin{equation}
\hat{\mathcal{B}}_{\bf u}\,p=-p'-\mu\left(\cdot,F_1({\bf u})(\cdot)\right)p, \quad D(\hat{\mathcal{B}}_{\bf u})=\left\{p\in W^{1,1}(0,m)\,\,|\,\, p(0)=0\right\}, \label{heq7}
\end{equation}
and introduce the operator
\begin{equation}
\hat{\mathcal{C}}_{\bf u}\,p=-\Phi_{\bf u}(p)\,\left(\hat{\mathcal{B}}_{\bf u}\right)_{-1}\left(\exp\left\{-\int_0^\cdot\mu(r,F_1({\bf u})(r))\,\ud r\right\}\right), \label{heq8}
\end{equation}
where $\left(\hat{\mathcal{B}}_{\bf u}\right)_{-1}$ is the unique
continuous extension of  $\hat{\mathcal{B}}_{\bf u}$ with domain
$D\left(\left(\hat{\mathcal{B}}_{\bf u}\right)_{-1}\right)=\mathcal{X}$.
Note that $\hat{\mathcal{C}}_{\bf u}$ is a bounded linear operator from
$\mathcal{X}$ to $\mathcal{X}_{-1}$, and in fact it is of rank $1$. For convenience we
also introduce the notation
\begin{equation} \label{pi}
\pi(a):=\pi\left(a,F_1({\bf
u})\right):=\exp\left\{-\int_0^a\mu(r,F_1({\bf u})(r))\,\ud
r\right\},\quad a\in [0,m].
\end{equation}

\begin{proposition} \label{proppart}
For every ${\bf u}\in\mathcal{Y}^0_+$, the operator $\mathcal{B}_{\bf u}$ is the part of the operator \\ $\left(\left(\hat{\mathcal{B}}_{\bf u}\right)_{-1}+\hat{\mathcal{C}}_{\bf u}\right)$ in $\mathcal{X}$.
\end{proposition}
\begin{proof} First we show that the domain of the part of
$\left(\left(\hat{\mathcal{B}}_{\bf u}\right)_{-1}+\hat{\mathcal{C}}_{\bf u}\right)$
in $\mathcal{X}$ coincides with the domain of $\mathcal{B}_{\bf u}$.
That is, we need to show that
\begin{equation}
D_1:=\left\{y\in\mathcal{X}\,\left\vert\,
\left(\left(\hat{\mathcal{B}}_{\bf u}\right)_{-1}+\hat{\mathcal{C}}_{\bf u}\right)\right.\,y\in\mathcal{X}\right\}=\left\{y\in
W^{1,1}(0,m)\,|\,y(0)=\Phi_{\bf u}(y)\right\}=:D_2.\label{heq9}
\end{equation}

First assume that $y\in D_2$. We define the following sequences of functions

\begin{align}
y_n(x):= &
\begin{Bmatrix}
y(x),\quad x> \frac{1}{n} \\
y\left(\frac{1}{n}\right)nx,\quad x\le\frac{1}{n}
\end{Bmatrix}, \quad n\in\mathbb{N},\quad x\in [0,m], \\
\pi_n(x):= &
\begin{Bmatrix}
\pi(x),\quad x>\frac{1}{n} \\
\pi\left(\frac{1}{n}\right)nx,\quad x\le\frac{1}{n}
\end{Bmatrix}, \quad n\in\mathbb{N},\quad x\in [0,m].
\end{align}

Note that $y_n(0)=\pi_n(0)=0,\,\forall\,n\in\mathbb{N}$, and $y_n\to y,\, \pi_n\to\pi$ in $\mathcal{X}$, moreover $\forall\, n\in\mathbb{N},\,\, y_n,\pi_n\in D(\hat{\mathcal{B}}_{\bf u})$.

Since
\begin{equation}
\left(\hat{\mathcal{B}}_{\bf u}\right)_{-1} y_n \to \left(\hat{\mathcal{B}}_{\bf u}\right)_{-1}y,\,\, \text{and}\,\,  -\Phi_{\bf u}(y_n)\left(\hat{\mathcal{B}}_{\bf u}\right)_{-1}\pi_n \to -\Phi_{\bf u}(y)\left(\hat{\mathcal{B}}_{\bf u}\right)_{-1}\,\pi=\hat{\mathcal{C}}_{\bf u}\,y,
\end{equation}
in  $\mathcal{X}_{-1}$; we have that
\begin{equation}
\left(\hat{\mathcal{B}}_{\bf u}\right)_{-1} y_n-\Phi_{\bf u}(y_n)\left(\hat{\mathcal{B}}_{\bf u}\right)_{-1}\pi_n \to\left(\left(\hat{\mathcal{B}}_{\bf u}\right)_{-1}+\hat{\mathcal{C}}_{\bf u}\right)\,y,\quad \text{in} \quad \mathcal{X}_{-1}.
\end{equation}
But we are going to show that in fact
\begin{equation}
\left(\hat{\mathcal{B}}_{\bf u}\right)_{-1}y_n-\Phi_{\bf u}(y_n)\left(\hat{\mathcal{B}}_{\bf u}\right)_{-1}\pi_n \to
\left(\left(\hat{\mathcal{B}}_{\bf u}\right)_{-1}+\hat{\mathcal{C}}_{\bf u}\right)\,y
= -y'-\mu y, \quad \text{in} \quad \mathcal{X}.
\end{equation}

To this end note that for $x>\frac{1}{n}$ we have (below for simplicity we are suppressing
the arguments in the functions $y$, $\mu$ and $\pi$)

\begin{equation}
\left(\hat{\mathcal{B}}_{\bf u}\right)_{-1} y_n-\Phi_{\bf u}(y_n)\left(\hat{\mathcal{B}}_{\bf u}\right)_{-1}\pi_n =-y'_n-\mu\,y_n+\Phi_{\bf u}(y_n)\left(\pi'_n+\mu\pi_n\right)
= -y'-\mu\,y,
\end{equation}
since $\pi'_n(x)+\mu\pi_n(x)=0,\,\,\forall\,x>\frac{1}{n}$.

While for $x\le\frac{1}{n}$ we have
\begin{align}
 \left(\hat{\mathcal{B}}_{\bf u}\right)_{-1}y_n  & - \Phi_{\bf u}(y_n)\left(\hat{\mathcal{B}}_{\bf u}\right)_{-1}\pi_n  \\
  = & -n\,y\left(\frac{1}{n}\right)-\mu\,n\,y\left(\frac{1}{n}\right)\,x+\Phi_{\bf u}(y_n)\left(n\,\pi\left(\frac{1}{n}\right)+\mu\,n\,\pi\left(\frac{1}{n}\right)\,x\right).
\end{align}

Therefore we obtain

\begin{align}
& \left|\left| \left(\hat{\mathcal{B}}_{\bf u}\right)_{-1}y_n-\Phi_{\bf u}(y_n)\left(\hat{\mathcal{B}}_{\bf u}\right)_{-1}\pi_n  -(-y'-\mu \,y) \right|\right|_{\mathcal{X}} \le \int_0^{\frac{1}{n}} |y' + \mu \, y| \,\ud x \nonumber \\
& + \int_0^{\frac{1}{n}}\left|\Phi_{\bf u}(y_n)\,n\,\pi\left(\frac{1}{n}\right)-n\,y\left(\frac{1}{n}\right)\right|\,\ud x
+\int_0^{\frac{1}{n}}\left|n\,x\left(\Phi_{\bf u}(y_n)\mu\,\pi\left(\frac{1}{n}\right)-\mu\,y\left(\frac{1}{n}\right)\right)\right|\,\ud x \nonumber \\
& \le \int_0^{\frac{1}{n}} |y' + \mu \, y| \,\ud x +
\left|\Phi_{\bf u}(y_n)\pi\left(\frac{1}{n}\right)-y\left(\frac{1}{n}\right)\right|+\frac{\sup(\mu)}{2n}\,\left|\Phi_{\bf u}(y_n)
  \pi\left(\frac{1}{n}\right)-y\left(\frac{1}{n}\right)\right| \nonumber \\ 
  & \xrightarrow[n\to\infty]{}0,
\end{align}
since $y$ and $y'$ belong to $L^1$, $\Phi_{\bf u}(y_n)\rightarrow
\Phi_{\bf u}(y) = y(0)$ ($y_n$ tends to $y$ in $\mathcal{X}$),
$y\left(\frac{1}{n}\right) \rightarrow y(0)$ (note that $y$ is a continuous function), and
$\pi\left(\frac{1}{n}\right) \rightarrow 1$, as $n\to\infty$. Hence we conclude that $y\in D_1$.

Next assume that $y\in D_1$ and take an arbitrary sequence of smooth
functions $y_n$, such that $y_n \in W^{1,1}(0,m)$ and $y_n(0)=0$ for
every $n\in\mathbb{N}$, and such that $y_n\to y$ in $\mathcal{X}.$
Then we have

\begin{equation}
\left(\left(\hat{\mathcal{B}}_{\bf u}\right)_{-1}+\hat{\mathcal{C}}_{\bf u}\right)\,y_n\xrightarrow[n\to\infty]{}\left(\left(\hat{\mathcal{B}}_{\bf u}\right)_{-1}+\hat{\mathcal{C}}_{\bf u}\right)\,y=z\in\mathcal{X},\label{heq10}
\end{equation}
with convergence in $\mathcal{X}_{-1}$. Using the definition of the $\mathcal{X}_{-1}$ norm it is shown that \eqref{heq10} is equivalent to

\begin{align}
\left(\hat{\mathcal{B}}_{\bf u}\right)^{-1}\ \left(
\left(\hat{\mathcal{B}}_{\bf u}\right)_{-1}+\hat{\mathcal{C}}_{\bf u}\right) y_n
 & = y_n + \left(\hat{\mathcal{B}}_{\bf u}\right)_{-1}^{-1} \hat{\mathcal{C}}_{\bf u}
\,y_n =
y_n - \left(\hat{\mathcal{B}}_{\bf u}\right)^{-1}_{-1}\Phi_{\bf u}(y_n)\left(\hat{\mathcal{B}}_{\bf u}\right)_{-1}\pi \nonumber \\
& = y_n-\Phi_{\bf u}(y_n)\pi\to \left(\hat{\mathcal{B}}_{\bf u}\right)^{-1}\,z\ =: w \in
 D(\hat{\mathcal{B}}_{\bf u}), \quad \text{in}\,\,\mathcal{X};
\end{align}
i.e. $w \in W^{1,1}(0,m)$, and $w(0)=0.$

Since $y_n\xrightarrow[\mathcal{X}]{} y$ and $\Phi_{\bf u}(y_n) \to \Phi_{\bf u}(y),$ we
have $w=y-\Phi_{\bf u}(y) \pi$, which implies that $y \in W^{1,1}(0,m)$, and
$y(0)=w(0)+\Phi_{\bf u}(y)\pi(0) = \Phi_{\bf u}(y)$; that is $y \in D_2.$

Now, for $y\in D_1,$ we have
\begin{align}
\left(\left(\hat{\mathcal{B}}_{\bf u}\right)_{-1}+\hat{\mathcal{C}}_{\bf u}\right)\,y
& = \left(\hat{\mathcal{B}}_{\bf u}\right)_{-1}\,y-\Phi_{\bf u}(y)\left(\hat{\mathcal{B}}_{\bf u}\right)_{-1}\pi   \\
&  =\left(\hat{\mathcal{B}}_{\bf u}\right)_{-1}\left(y - \Phi_{\bf u}(y)\pi
\right) =    \left(\hat{\mathcal{B}}_{\bf u}\right)_{-1}\left(y - y(0)\pi
\right)  \\
& = -(y'-y(0)\pi')-\mu(y-y(0)\pi) = -y' -\mu y = \mathcal{B}_{\bf u}\,y,
\end{align}
where we used that $(y - y(0)\pi) \in D\left(\hat{\mathcal{B}}_{\bf u}\right)$, and that $y \in
D_2 = D_1.$ Hence the proof is completed. 

\end{proof}

From the point of view of applying our steady state framework to model \eqref{heq1}-\eqref{heq2}, the significance of recovering the generator $\mathcal{B}_{\bf u}$ as part of
$\left(\left(\hat{\mathcal{B}}_{\bf u}\right)_{-1}+\hat{\mathcal{C}}_{\bf u}\right)$
is that, since the semigroups $\mathcal{T}(t)$ and
$\mathcal{T}_{-1}(t)$ are similar, that is
\begin{equation}
\mathcal{T}(t)=\left(\left(\hat{\mathcal{B}}_{\bf u}\right)_{-1}+\hat{\mathcal{C}}_{\bf u}\right)^{-1}\,\mathcal{T}_{-1}(t)\,\left(\left(\hat{\mathcal{B}}_{\bf u}\right)_{-1}+\hat{\mathcal{C}}_{\bf u}\right),\quad t\ge 0,
\end{equation}
holds; their spectra coincide, see \cite[Ch.II]{NAG} for more details. 

This in particular means that we can study the behaviour of $s\left(\left(\hat{\mathcal{B}}_{\bf u}\right)_{-1}+\hat{\mathcal{C}}_{\bf u}\right)$ as a function of ${\bf u}$, instead of studying directly the behaviour of  $s(\mathcal{B}_{\bf u})$ as a function of the parameter ${\bf u}$, and vice versa.

Note that if we define the positive cone of $\mathcal{X}_{-1}$ as the completion of $\mathcal{X}_+$ with respect to the norm $||\cdot||_{-1}$, then it is shown that $\mathcal{X}_{-1}$, with the usual partial ordering, is a Banach lattice itself.

We now formulate conditions which guarantee the existence of a (strictly) positive steady state of the nonlinear model \eqref{heq1}-\eqref{heq2}.

\begin{theorem}\label{result1}
Assume that $\beta$ is a strictly monotone decreasing function with respect to its second variable,  and that
\begin{enumerate}
\item $\mu>\mu_0$, for some $\mu_0>0$.
\item  $\exists\, K>0\quad \text{such that}\quad \displaystyle\max_{a\in[0,m]}\,\{\beta(a,K)\}<\mu_0$.
\item $\beta(m,x)>0,\,\,\forall\,x\in\mathbb{R}_+$.
\end{enumerate}
Then, if there exists an $r>0$, such that $s(\mathcal{B}_{\bf u})>0$ holds for all ${\bf u}\in\mathcal{Y}^0_+$ ,  $||{\bf u}||\le r$, then model \eqref{heq1}-\eqref{heq2} admits a strictly positive steady state.
\end{theorem}

\begin{proof}  Similarly as we have applied our framework previously e.g. in \cite{CF} for a model with infinite dimensional nonlinearity, but with distributed recruitment process; our goal is to transform the steady
state problem into a fixed point problem in the parameter space.
This requires us to establish some desirable properties of the
level set $S$ (defined in \eqref{levelset}). In particular, we are
going to establish the following properties.
\begin{itemize}
\item[(i)]  $s\left(\mathcal{B}_{\bf u}\right)$ is a continuous function of the parameter ${\bf u}$.
\item[(ii)] $s\left(\mathcal{B}_{\bf u}\right)$ is strictly monotone decreasing (as a function of ${\bf u}$) along rays in $\mathcal{Y}_+^0$.
\item[(iii)] The zero level set $S\subset\mathcal{Y}_+^0$ defined in \eqref{levelset} is contained in a bounded ball.
\end{itemize}

To establish the continuity of $s\left(\mathcal{B}_{\bf u}\right)$ as a function of the parameter ${\bf u}$, previously (see  \cite{CF2}) we used the notion of generalised convergence of operators and perturbation results from \cite{K}. Here, in contrast, we employ a more direct approach, noting that  $\mathcal{B}_{\bf u}$ generates an eventually compact positive semigroup, therefore its spectral bound $s\left(\mathcal{B}_{\bf u}\right)$ is an isolated eigenvalue, for every ${\bf u}\in \mathcal{Y}_+^0$. In fact $s\left(\mathcal{B}_{\bf u}\right)$ is the dominant real solution of the characteristic equation
\begin{equation}
1=\int_0^m\beta(a,F_2({\bf
u}))\exp\left\{-\int_0^a\left(\lambda+\mu(r,F_1({\bf
u})(r))\right)\,\ud r\right\}\,\ud a=:K({\bf u},\lambda).
\end{equation}
We now set 
\begin{equation}\label{G-def}
K({\bf u},\lambda)-1=:G({\bf u},\lambda)\,:\, \left\{\mathcal{Y}\setminus\left\{{\bf u} \in \mathcal{Y}:\int_0^m {\bf u}(r) \,\ud r = 0 \right\}\right\}\times\mathbb{R}\to\mathbb{R},
\end{equation} 
and apply the Implicit Function Theorem. In particular we use Theorem A from Appendix A in
\cite{CR}. Note that if we naturally extend the functions $\beta$
and $\mu$ as $\beta(\cdot,x)=\beta(\cdot,0),\,\forall x<0$, and
$\mu(\cdot, x):=\mu(\cdot, 0),\,\forall\, x < 0$, then $G$ as defined above in \eqref{G-def} (on an open set) is a continuous map.

Also note that for any
$\left({\bf u}_*,s(\mathcal{B}_{{\bf u}_*})\right)\in
\left\{\mathcal{Y}\setminus\left\{{\bf u} \in \mathcal{Y}:\int_0^m {\bf u}(r)
\,\ud r = 0 \right\}\right\}\times\mathbb{R}$ we have 
\begin{equation*}
G\left({\bf u}_*,s(\mathcal{B}_{{\bf u}_*})\right)=0,
\end{equation*} 
and the map $\lambda\to G({\bf u},\lambda)$ is continuously differentiable for every parameter value ${\bf
u}\in \mathcal{Y}\setminus\left\{{\bf u} \in \mathcal{Y}:\int_0^m {\bf u}(r)
\,\ud r = 0 \right\}$. 

In fact, since
\begin{equation}
\frac{\partial G}{\partial \lambda}({\bf u},\lambda)=-\int_0^m a\beta(a,F_2({\bf u}))\exp\left\{-\int_0^a\left(\lambda+\mu(r,F_1({\bf u})(r))\right)\,\ud r\right\}\,\ud a\neq 0,
\end{equation}
we have that the map $x\to \frac{\partial G}{\partial
\lambda}\left({\bf u}_*,s(\mathcal{B}_{{\bf u}_*})\right)x$ is a
linear homeomorphism from $\mathbb{R}$ onto $\mathbb{R}$.

Then, the Implicit Function Theorem implies that
$s(\mathcal{B}_{{\bf u}_*})$ is a continuous function of ${\bf u}_*$
for all ${\bf u}_*\in \mathcal{Y}\setminus\left\{{\bf u} \in
\mathcal{Y}:\int_0^m {\bf u}(r) \,\ud r = 0 \right\}$.

Next we show that the function ${\bf u}\to s(\mathcal{B}_{\bf u})$ is strictly monotone decreasing along positive rays in $\mathcal{Y}$.  Using the extended generator, we have for every ${\bf u}\in\mathcal{Y}_+^0$, $v\in\mathcal{X}_+$,
 and $0<\alpha_1<\alpha_2$ (note that ${\bf u}$ is a parameter,
while $v$ is an element of the positive cone of the state space
$\mathcal{X}$)
\begin{equation}\label{opdiff}
\left(\left(\left(\hat{\mathcal{B}}_{\alpha_1 {\bf
u}}\right)_{-1}+\hat{\mathcal{C}}_{\alpha_1 {\bf u}}\right)-
\left(\left(\hat{\mathcal{B}}_{\alpha_2{\bf
u}}\right)_{-1}+\hat{\mathcal{C}}_{\alpha_2{\bf
u}}\right)\right)\,v=\left(\hat{\mathcal{C}}_{\alpha_1{\bf u}}-
\hat{\mathcal{C}}_{\alpha_2{\bf u}}\right)\,v\ge 0,
\end{equation}
a positive perturbation. 

Indeed, using that $F_1$ is invariant along rays and
the notation \eqref{pi}, we have
\begin{align*}
\left(\hat{\mathcal{C}}_{\alpha_1{\bf u}}-
\hat{\mathcal{C}}_{\alpha_2{\bf u}}\right)\,v & = \left(
\Phi_{\alpha_1{\bf u}}(v) - \Phi_{\alpha_2{\bf u}}(v) \right)
\left(-\left(\hat{\mathcal{B}}_{{\bf u}}\right)_{-1} \pi\right) \\
 & = \int_0^m \left(\beta(a,F_2(\alpha_1 {\bf u}))-
\beta(a,F_2(\alpha_2 {\bf u}))\right) v(a) da
\left(-\left(\hat{\mathcal{B}}_{{\bf u}}\right)_{-1} \pi\right) \geq
0.
\end{align*}

This is because $F_2(\alpha_1 {\bf u})-F_2(\alpha_2 {\bf u}) =
(\alpha_1-\alpha_2) \int_0^m {\bf u}(a) da < 0$, and hence the integral above is
positive because $\beta$ decreases with respect to its second
variable and, moreover, we have $\left(-\left(\hat{\mathcal{B}}_{{\bf
u}}\right)_{-1}\pi\right) \in \left(\mathcal{X}_{-1}\right)_{+}$, because
\begin{equation*}
\left(-\left(\hat{\mathcal{B}}_{{\bf u}}\right)_{-1} \right)\pi =\lim_{n\to\infty} \left(-\hat{\mathcal{B}}_{{\bf u}}\right)
\pi_n =  \lim_{n\to\infty} \left(n\pi\left(\frac{1}{n}\right)(1+\mu\,x)\chi_{[0,\frac{1}{n}]}(x)\right),
\end{equation*} 
where the limits above are in $\mathcal{X}_{-1}$, (see also the proof of Prop. \ref{proppart}).

Note that
$\left(\hat{\mathcal{B}}_{\alpha_1{\bf
u}}\right)_{-1}-\left(\hat{\mathcal{B}}_{\alpha_2{\b
u}}\right)_{-1}=0$ follows from the fact that the parts of
$\left(\hat{\mathcal{B}}_{\alpha_1{\bf u}}\right)_{-1}$ and
$\left(\hat{\mathcal{B}}_{\alpha_2{\bf u}}\right)_{-1}$ in
$\mathcal{X}$ satisfy $\hat{\mathcal{B}}_{\alpha_1{\bf
u}}=\hat{\mathcal{B}}_{\alpha_2{\bf u}}$, and the only unique
continuous extension of the zero operator from a dense subset to
$\mathcal{X}$ is the zero operator.

As we have noted before, $\mathcal{B}_{\bf u}$ generates a positive
and eventually compact semigroup of operators, which is also
irreducible under our hypotheses, and therefore the spectral bound
$s(\mathcal{B}_{\bf u})$ belongs to its spectrum and it is an
eigenvalue of algebraic multiplicity one; and as a consequence
the same holds for $s\left(\left(\hat{\mathcal{B}}_{\bf
u}\right)_{-1}+\hat{\mathcal{C}}_{\bf u}\right)$.

Also note that for every sufficiently large $\lambda$, we have the identity
\begin{align}
 & R\left(\lambda,\left(\left(\hat{\mathcal{B}}_{\alpha_1{\bf u}}\right)_{-1}+\hat{\mathcal{C}}_{\alpha_1{\bf u}}\right)\right)-R\left(\lambda,\left(\left(\hat{\mathcal{B}}_{\alpha_2{\bf u}}\right)_{-1}+\hat{\mathcal{C}}_{\alpha_2{\bf u}}\right)\right) \nonumber \\
  = &  R\left(\lambda,\left(\left(\hat{\mathcal{B}}_{\alpha_1{\bf u}}\right)_{-1}+\hat{\mathcal{C}}_{\alpha_1{\bf u}}\right)\right)
 \left(\hat{\mathcal{C}}_{\alpha_1{\bf u}}-\hat{\mathcal{C}}_{\alpha_2{\bf u}}\right)R\left(\lambda,\left(\left(\hat{\mathcal{B}}_{\alpha_2{\bf u}}\right)_{-1}+\hat{\mathcal{C}}_{\alpha_2{\bf u}}\right)\right),
\end{align}
and therefore Proposition A.2 in \cite{AB} (see also Theorem 1.3 in \cite{AB2}) implies that
\begin{equation}
s\left(\left(\hat{\mathcal{B}}_{\alpha_1{\bf u}}\right)_{-1}+\hat{\mathcal{C}}_{\alpha_1{\bf u}}\right)>
s\left(\left(\hat{\mathcal{B}}_{\alpha_2{\bf u}}\right)_{-1}+\hat{\mathcal{C}}_{\alpha_2{\bf u}}\right),
\end{equation}
which indeed shows that the function ${\bf u}\to s(\mathcal{B}_{\bf u})$ is strictly monotone decreasing along positive rays in $\mathcal{Y}$.

Finally we are going to prove  that the zero level set $S$ is contained in a bounded ball. To this end we let
\begin{equation*}
\lambda_{\bf u}:=s\left(\left(\hat{\mathcal{B}}_{\bf u}\right)_{-1}+\hat{\mathcal{C}}_{\bf u}\right)=s(\mathcal{B}_{\bf u}),\quad {\bf u}\in\mathcal{Y}_+^0.
\end{equation*}
We are going to show that for all ${\bf u}\in\mathcal{Y}_+^0$ such that $||{\bf u}||>K$ holds, we have $\lambda_{\bf u}<0$, which implies the assertion that $S$ is contained in a bounded ball.

Note that for every ${\bf u}\in\mathcal{Y}_+^0$, $\lambda_{\bf u}$ is a dominant real eigenvalue of multiplicity one with a corresponding strictly positive eigenvector $p_{\bf u}$, satisfying the equation
\begin{equation}\label{spectrb1}
\left(\left(\hat{\mathcal{B}}_{\bf u}\right)_{-1}+\hat{\mathcal{C}}_{\bf u}\right)p_{\bf u}=\lambda_{\bf u}\,p_{\bf u},
\end{equation}
which is equivalent to
\begin{equation}\label{spectrb2}
\left(\left(\hat{\mathcal{B}}_{\bf u}\right)_{-1}-\lambda_{\bf u}\,\mathcal{I}\right)p_{\bf u}=\Phi_{\bf u}(p_{\bf u})\,\left(\hat{\mathcal{B}}_{\bf u}\right)_{-1}\left(\exp\left\{-\int_0^\cdot\mu(r,F_1({\bf u})(r))\,\ud r\right\}\right).
\end{equation}

Note that the definition of $\hat{\mathcal{B}}_{\bf u}$ (in particular the homogeneous boundary condition) implies that
\begin{equation}\label{spectrb3}
s\left(\left(\hat{\mathcal{B}}_{\bf u}\right)_{-1}\right)= - \infty,
\end{equation}
which means that for any $\lambda$ the operator $\left(\left(\hat{\mathcal{B}}_{\bf u}\right)_{-1}-\lambda\,\mathcal{I}\right)$ is invertible. Hence from \eqref{spectrb2} we obtain
\begin{equation}\label{spectrb4}
\Phi_{\bf u}(p_{\bf u})=\Phi_{\bf u}(p_{\bf u})\,\Phi_{\bf u}\left( \left(\left(\hat{\mathcal{B}}_{\bf u}\right)_{-1}-\lambda_{\bf u}\,\mathcal{I}\right)^{-1} \left[\left(\hat{\mathcal{B}}_{\bf u}\right)_{-1}\left(\exp\left\{-\int_0^\cdot\mu(r,F_1({\bf u})(r))\,\ud r\right\}\right) \right] \right),
\end{equation}
which yields (since $p_{\bf u}$ is strictly positive)
\begin{equation}\label{spectrb5}
1=\Phi_{\bf u}\left( \left(\left(\hat{\mathcal{B}}_{\bf u}\right)_{-1}-\lambda_{\bf u}\,\mathcal{I}\right)^{-1} \left[\left(\hat{\mathcal{B}}_{\bf u}\right)_{-1}\left(\exp\left\{-\int_0^\cdot\mu(r,F_1({\bf u})(r))\,\ud r\right\}\right) \right] \right)=:F(\lambda_{\bf u}).
\end{equation}

Note that $F$ is monotone decreasing for $\lambda_{\bf u}>0$, and at $\lambda_{\bf u}=0$ we have
\begin{align}
F(0)= & \Phi_{\bf u}\left(\exp\left\{-\int_0^\cdot\mu(r,F_1({\bf u})(r))\,\ud r\right\}\right) \\ \nonumber
= &\int_0^m\beta(a,F_2({\bf u}))\exp\left\{-\int_0^a\mu(r,F_1({\bf u})(r))\,\ud r\right\}\,\ud a \\ \nonumber
< & \int_0^m \mu_0 e^{-\mu_0 a}\,\ud a<1, \label{spectrb6}
\end{align}
for every ${\bf u}\in\mathcal{Y}_+^0$ satisfying $||{\bf u}||>K$, which shows that the spectral bound indeed changes sign in a bounded ball along every ray in $\mathcal{Y}^0_+$, and hence $S$ itself is contained in a bounded ball.

We also note that the fact that $S$ is bounded away from the origin follows directly from our last assumption in Theorem \ref{result1}.

To complete the proof let us denote by $B_+$ the unit sphere of $\mathcal{Y}$ intersected with the positive cone $\mathcal{Y}_+$, and define a map $\Theta$ from the closed and convex set $B_+$ into itself as:
\begin{equation}
\Theta\,:\,\underbrace{{\bf n}}_{\in B_+}\xrightarrow{h^{-1}}\underbrace{{\bf u}}_{\in S}\quad\xrightarrow{g}\underbrace{\mathcal{B}_{\bf u}}_{\in g(S)\subset C(\mathcal{X})}
\xrightarrow{e}\underbrace{{\bf V}_{\bf u}}_{\in W^{1,1}_+(0,m)}
\xrightarrow{k}\underbrace{{\bf V}_{\bf u}}_{\in L^1_+(0,m)}
\xrightarrow{p}\underbrace{{\bf u'}}_{\in S}\xrightarrow{h}\underbrace{{\bf n'}}_{\in B_+}.
\end{equation}
Above, $h$, defined as $h({\bf u})=\frac{{\bf u}}{||{\bf u}||}$, is continuous with a continuous inverse (this is because $S$ is bounded away from the origin and it is contained in a bounded ball). The projection $p$ along positive rays in $L_+^1$ is also continuous and bounded. $k$ is the compact injection  of $W^{1,1}(0,m)$ into $L^1(0,m)$. The map $e$, which assigns the strictly positive normalised eigenvector ${\bf V}_{\bf u}$ of $\mathcal{B}_{\bf u}$ to $\mathcal{B}_{\bf u}$, is analytic, (see e.g. \cite[Lemma 1.3]{CR1973}). Finally, the map $g$ above, which assigns the linear operator $(\mathcal{B}_{\bf u},D(\mathcal{B}_{\bf u}))\in C(\mathcal{X})$ (the set of closed operators on $\mathcal{X}$) to the parameter value ${\bf u}\in S$, is shown to be continuous using the notion of generalised convergence, see \cite[Sect.2 Ch. IV]{K}. (Note that compared to e.g. \cite{CF,CF3} the situation is more delicate, as both the operator and its domain depend on the parameter value ${\bf u}$.) In particular, continuity is established via \cite[Th.2.25 Ch. IV.]{K}, which characterizes generalised convergence of operators via convergence of their resolvents.  Since the resolvent of $\mathcal{B}_{\bf u}$ is explicitly given as ($f\in L^1(0,m)$)
\begin{align}
& \mathcal{R}(\mathcal{B}_{\bf u},\lambda) \,f  = \left[\frac{\int_0^m\beta(a,F_2({\bf u}))\int_0^a f(x)\exp\left\{-\int_x^a(\lambda+\mu(r,F_1({\bf u})(r)))\,\ud r\right\}\,\ud x\,\ud a}{1-\int_0^m\beta(a,F_2({\bf u}))\exp\left\{-\int_0^a(\lambda+\mu(r,F_1({\bf u})(r)))\,\ud r\right\}\,\ud a } \right. \nonumber \\
+ & \left.  \int_0^a f(x)\exp\left\{\int_0^x(\lambda+\mu(r,F_1({\bf u})(r)))\,\ud r\right\}\,\ud x\right]\exp\left\{-\int_0^a\left(\lambda+\mu(r,F_1({\bf u})(r))\right)\,\ud r\right\},
\end{align} it is shown, by applying Lebesgue's dominated convergence theorem, that for any sequence $\left\{{\bf u_n}\right\}\subset S$ tending to ${\bf \hat{u}}\in S$ there exists a $\lambda$ large enough, such that
\begin{equation}
\left|\left|\mathcal{R}(\mathcal{B}_{\bf u_n},\lambda)-\mathcal{R}(\mathcal{B}_{\bf \hat{u}},\lambda)\right|\right|\to 0,
\end{equation} 
as $n\to \infty$. \\ By Schauder's fixed point theorem the map $\Theta$ has a fixed point. Equivalently, the map $h^{-1}\circ\Theta\circ h$, which maps $S$ into itself, has a fixed point ${\bf u}_*$ which turns out to be a
steady state of the problem:
\begin{align*}
\mathcal{B}_{\bf u^*}\,{\bf u^*} & = \mathcal{B}_{\bf u^*}\left(\left(h^{-1}\circ\Theta\circ h\right){\bf u}_*\right)=\mathcal{B}_{{\bf u}_*}\left(\left(p\circ k\circ e\circ g\right) {\bf u}_*\right) \\ 
& =\mathcal{B}_{{\bf u}_*}\left(\alpha\left({\bf V}_{{\bf u}_*}\right)\left( k\circ e\circ g\right) {\bf u}_*\right)= \alpha\left({\bf V}_{{\bf u}_*}\right) \mathcal{B}_{{\bf u}_*}\left(\left( k\circ e\circ g\right) {\bf u}_*\right) \\
& = \alpha\left({\bf V}_{{\bf u}_*}\right) \mathcal{B}_{{\bf u}_*}\left({\bf V}_{{\bf u}_*}\right)= 0,
\end{align*}
for some $0\ne \alpha({\bf V}_{{\bf u}_*})\in\mathbb{R}$; since ${\bf V}_{{\bf u}_*}$ belongs to the kernel of $\mathcal{B}_{{\bf u}_*}$.
\end{proof}

The following (counterpart) result can be established along similar  lines.
\begin{theorem}\label{result2}
Assume that $\beta$ is a strictly monotone increasing function with respect to its second variable,  and that
\begin{enumerate}
\item $\beta(m,x)>0,\,\,\forall\,x\in\mathbb{R}_+$.
\item There exists an $r>0$, such that $s(\mathcal{B}_{\bf u})<0$ holds for all ${\bf u}\in\mathcal{Y}_+^0$ ,  $||{\bf u}||\le r$.
\item There exists an $R>r>0$, such that $s(\mathcal{B}_{\bf u})>0$ holds for all ${\bf u}\in\mathcal{Y}_+^0$ ,  $||{\bf u}||\ge R$.
\end{enumerate}
Then, model \eqref{heq1}-\eqref{heq2} admits a strictly positive steady state.
\end{theorem}

As we have already noted in \cite{CF2}, the crucial condition that the spectral bound is monotone (and changes sign) along positive rays of the parameter space is closely related to the biologically meaningful condition that an appropriately defined net reproduction function, or in our case a net reproduction functional, changes mo\-no\-tonously (and crosses the value $1$) along positive rays of the parameter space. We now further explore this connection, and in particular we recover the familiar net reproduction number/function for age-structured (single state at birth) models.

To this end, for every fixed ${\bf u}\in\mathcal{Y}_+^0$ we define the net reproduction number as
\begin{equation}
\mathcal{R}({\bf u})=\int_0^m\beta(a,F_2({\bf u}))\exp\left\{-\int_0^a\mu(r,F_1({\bf u})(r))\,\ud r\right\}\,\ud a,\label{netrep}
\end{equation}
(notice that  $\mathcal{R}({\bf u})$ coincides with $F(0)$  as defined in \eqref{spectrb5}), 
and for ${\bf u}\equiv {\bf 0}$ (i.e. the extinction state/environment) we may naturally define the net reproduction number as
\begin{equation}
\mathcal{R}({\bf 0})=\int_0^m\beta(a,0)\exp\left\{-\int_0^a\mu(r,{\bf 0})\,\ud r\right\}\,\ud a.\label{netrep2}
\end{equation}

Note that naturally $\mathcal{R}\,:\,\mathcal{Y}_+^0\to\mathbb{R}_+$ can be considered as a functional, and so we may call it the net reproduction functional, see also \cite{Farkas2018} for more details in this direction.
Also note that  for the specific model we considered here we had to formally define $\mathcal{R}$ at the extinction steady state (zero population density) separately, since $F_1$ is not defined at ${\bf 0}$ in general. Note that in general the value $\mathcal{R}({\bf 0})$ is often denoted by simply $\mathcal{R}$ or $\mathcal{R}_0$ is usually referred to as the net reproduction number or basic reproduction number in the literature. We would like to emphasize though that for nonlinear models one really has to keep in mind that these scalar values (for example $\mathcal{R}_0$) are simply the values of a net reproduction function/functional at fixed population densities (environments). Indeed, net reproduction numbers, functions and functionals play an important role when studying the existence and the local asymptotic stability of equilibria of  physiologically structured population dynamical models, for more details see e.g. the recent papers \cite{BCR,CDF,CD,DGM,Farkas2018,FGH,FH,HT2,HT}.

We now rigorously establish the connection between the spectral bound $s(\mathcal{B}_{\bf u})$ and the net reproduction functional $\mathcal{R}({\bf u})$, as defined in \eqref{netrep}.
\begin{proposition}\label{reprod}
For every ${\bf u}\in\mathcal{Y}_+^0$ we have 
\begin{equation*}
\sigma\left(-\hat{\mathcal{C}}_{\bf u}\left(\left(\hat{\mathcal{B}}_{\bf u}\right)_{-1}\right)^{-1}\right)=\left\{0,\mathcal{R}({\bf u})\right\},
\end{equation*}
and in particular
\begin{equation}
s(\mathcal{B}_{\bf u})\gtrless 0 \iff R({\bf u})\gtrless 1.\label{spbound}
\end{equation}
\end{proposition}
\begin{proof}
To establish the relationship \eqref{spbound} for ${\bf u}\in\mathcal{Y}_+^0$ we apply Theorem 3.5 from \cite{HT} to the extended generator $\left(\left(\hat{\mathcal{B}}_{\bf u}\right)_{-1}+\hat{\mathcal{C}}_{\bf u}\right)$ in $\mathcal{X}_{-1}$.  Note that one can verify that the operators $\left(\hat{\mathcal{B}}_{\bf u}\right)_{-1},\,\hat{\mathcal{C}}_{\bf u}$ satisfy the assumptions of Theorem 3.5 in \cite{HT}.
Hence applying Theorem 3.5 from \cite{HT} we obtain
\begin{equation}
s(\mathcal{B}_{\bf u})=s\left(\left(\hat{\mathcal{B}}_{\bf u}\right)_{-1}+\hat{\mathcal{C}}_{\bf u}\right)\gtrless 0\iff r\left(-\hat{\mathcal{C}}_{\bf u}\left(\left(\hat{\mathcal{B}}_{\bf u}\right)_{-1}\right)^{-1}\right)\gtrless 1.
\end{equation}
It remains to show that
\begin{equation*}
r\left(-\hat{\mathcal{C}}_{\bf u}\left(\left(\hat{\mathcal{B}}_{\bf u}\right)_{-1}\right)^{-1}\right)=\mathcal{R}({\bf u}).
\end{equation*}
To this end note that the equation
\begin{align}
 \left(-\hat{\mathcal{C}}_{\bf u}\left(\left(\hat{\mathcal{B}}_{\bf u}\right)_{-1}\right)^{-1}\right)\,p & = \lambda\,p \\ \nonumber
& =\Phi_{\bf u}\left(\left(\left(\hat{\mathcal{B}}_{\bf u}\right)_{-1}\right)^{-1}p\right)\left[ \left(\hat{\mathcal{B}}_{\bf u}\right)_{-1}\left(\exp\left\{-\int_0^\cdot\mu(r,F_1({\bf u})(r))\,\ud r\right\}\right) \right],
\end{align}
shows that $\lambda=0$ is an eigenvalue with corresponding eigenvectors $p=\left(\hat{\mathcal{B}}_{\bf u}\right)_{-1}w$ (spanning an infinite dimensional subspace), such that $w\in\mathcal{X}$ and
\begin{equation*}
\int_0^m\beta(a,F_2({\bf u}))w(a)\,\ud a=0,
\end{equation*}
holds.

On the other hand if $\lambda\ne 0$ then we have that
\begin{equation*}
p=\alpha\left[\left(\hat{\mathcal{B}}_{\bf u}\right)_{-1}\left(\exp\left\{-\int_0^\cdot\mu(r,F_1({\bf u})(r))\,\ud r\right\}\right)\right]=:\alpha\,v,
\end{equation*}
for some $\alpha\ne 0$, which together with
\begin{equation*}
\lambda\,\alpha\,v=\Phi_{\bf u}\left(\left(\left(\hat{\mathcal{B}}_{\bf u}\right)_{-1}\right)^{-1}\alpha\,v\right)\,v=
\alpha \, \Phi_{\bf u}\left(\left(\left(\hat{\mathcal{B}}_{\bf u}\right)_{-1}\right)^{-1} v\right)\,v,
\end{equation*}
yields
\begin{equation*}
\lambda=\Phi_{\bf u}\left(\left(\left(\hat{\mathcal{B}}_{\bf u}\right)_{-1}\right)^{-1} v\right)=\Phi_{\bf u}\left(\exp\left\{-\int_0^\cdot\mu(r,F_1({\bf u})(r))\,\ud r\right\}\right).
\end{equation*}
Hence we have shown that (below $\sigma$ stands for the spectrum)
\begin{align}
\sigma\left(-\hat{\mathcal{C}}_{\bf u}\left(\left(\hat{\mathcal{B}}_{\bf u}\right)_{-1}\right)^{-1}\right) = \left\{0,\Phi_{\bf u}\left(\pi\right)\right\}=\left\{0,\mathcal{R}({\bf u})\right\},
\end{align}
which establishes the assertions of Proposition \ref{reprod}. 
\end{proof}

\begin{remark}
Note that by appropriately defining the operator $\mathcal{B}_{\bf 0}$ we can also establish that $s(\mathcal{B}_{\bf 0})\gtrless 0 \iff R({\bf 0})\gtrless 1$.
\end{remark}

\begin{remark}
We would like to emphasize that in the proof of Proposition \ref{reprod} we did not really use that $F_1,F_2$ assume the particular form as in \eqref{heq3}.
\end{remark}

\begin{remark}
We recall from \cite{GM} the classic nonlinear (Gurtin-McCamy) age-structured po\-pu\-lation dynamical model here.
\begin{align}
p_t(a,t)+p_a(a,t)= & -\mu\left(a,\int_0^m p(r,t)\,\ud r\right)p(a,t),\quad a\in (0,m), \label{GM1} \\
p(0,t)=& \int_0^m \beta\left(a,\int_0^m p(r,t)\,\ud r\right)p(a,t)\,\ud a, \quad t>0. \label{GM2}
\end{align}

Note that the proof of Proposition \ref{reprod} can be directly adapted for this model (with a one dimensional nonlinearity, i.e. $\mathcal{Y}=\mathbb{R}$, since the environment in\-di\-vi\-duals are experiencing is simply determined by the total population size $P(t)=\int_0^m p(a,t)\,\ud a$), and of course also for the basic linear age-structured McKendrick model (i.e. when $\beta$ and $\mu$ do not depend on the total population size $P(t)$). 

Thus we have  rigorously justified the biologically inspired definition of the net reproduction number
\begin{equation}
\mathcal{R}=\int_0^m\beta(a)\exp\left\{-\int_0^a \mu(r)\,\ud r\right\}\,\ud a,
\end{equation}
(appearing frequently in the literature, see e.g. \cite{Iannelli}) for the linear age-structured McKendrick model; and similarly the net reproduction function, which, for the Gurtin-McCamy (non-linear) age-structured model \eqref{GM1}-\eqref{GM2} reads as:
\begin{equation}
\mathcal{R}(P)=\int_0^m\beta(a,P)\exp\left\{-\int_0^a \mu(r,P)\,\ud r\right\}\,\ud a.
\end{equation}
Also note that analogously one arrives at the definition of the net reproduction function
\begin{equation}
\mathcal{R}(P)=\int_0^m\frac{\beta(s,P)}{\gamma(s,P)}\exp\left\{-\int_0^s \frac{\mu(r,P)}{\gamma(r,P)}\,\ud r\right\}\,\ud s,
\end{equation}
for the size-structured version of the Gurtin-McCamy model \eqref{GM1}-\eqref{GM2}, see e.g. \cite{FH}, where stability criteria for positive steady states were also established in terms of the derivative of the net reproduction function above.
\end{remark}

\section{Concluding remarks}

In this paper we have treated the (positive) steady state problem of a nonlinear structured population  model with a nonlinear and non-local boundary condition. In particular we have showed that the growth behaviour of the spectral bound of the parametrised family of linear operators arising in the steady state formulation can be controlled by boundary perturbations only. This way we have demonstrated how our general framework developed in \cite{CF2} can be applied to treat classes of nonlinear models with  single state at birth, and in particular with fertility/reproduction controlled dynamics. 

We note that for certain classes of models, a more direct fixed point approach as employed previously e.g. in \cite{FarkasHinow} can be also applied. The question is whether one formulates the fixed point map on the state space as e.g. in \cite{FarkasHinow}, or on the parameter (environment) space. Sometimes, in particular when it is not feasible to obtain an implicit formula for the steady state, it may not be possible to formulate a fixed point map on the state space. Another disadvantage of formulating the fixed point problem on the state space is that often fixed point theorems rely on contraction arguments, and so in the context of population models special care has to be taken to avoid the extinction steady state, as e.g. in \cite{FarkasHinow}, where a fixed point theorem on conical shells in Banach spaces was employed.

For these reasons in general we prefer to formulate a fixed point problem on the parameter space, and make use of the spectral properties of the family of operators arising from the parametrisation (using the environment as parameter), partly because then this has a natural biological intepretation, too. In fact, the goal of the present work was to show how the spectral bound of the family of operators can be controlled by boundary perturbations only, and the results here can be readily extended to nonlinear size-structured models with non-local boundary condition, see e.g. \cite{CF2} or even  \cite{CF3} (including a diffusion operator, too) in this direction.

The general idea of lifting a problem into the extrapolated space, where the boundary condition becomes a bounded perturbation of the generator is very powerful and can be applied to other models, as well. To name another concrete example, let us consider here the following general selection-mutation model (see also \cite{CP} for more details).
\begin{align}
u_t(l,a,t)+u_a(l,a,t) &=-\mu(E_1[u],l,a)u(l,a,t),\nonumber \\
u(l,0,t) &=\int_0^\infty\int_{\hat{l}}^\infty b(l,\hat{l})\beta(E_2[u],\hat{l},a)u(\hat{l},a,t)\, \ud a\,\ud\hat{l}, \label{selmut} \\
u(l,a,0) & =u_0(l,a). \nonumber
\end{align}
In the model above $u(l,a,t)$ denotes the density of individuals of age $a$ and maturation age $l$ at time $t$. $E_1$ and $E_2$ are the environmental operators, as described in Section 1. $\mu$ denotes the mortality rate, $\beta$ the fertility function and $b$ is the probability density function describing the mutation, that is $\int_{l_1}^{l_2} b(l,\hat{l})\, \ud l$ is the probability that the offspring of an individual of age at maturity $\hat{l}$ has maturity age $l\in (l_1,l_2)$. 
 
In \cite{CP} the existence and uniqueness of a positive steady state of model \eqref{selmut} was established, for the case when  $E_1\equiv E_2$, and when the mortality $\mu$ is a strictly monotone increasing function of its first variable, while the fertility $\beta$ is a strictly monotone decreasing function of its first variable, as well. In contrast, in \cite{CF2} we treated the case of $E_1\equiv E_2\equiv E[u(\cdot,\cdot,t)]=(P(t),Q(t))$ (i.e. two-dimensional nonlinearity),  where 
\begin{equation}
P(t)=\int_0^{a_m}\int_0^lu(l,a,t)\, \ud a\, \ud l,\quad Q(t)=\int_0^{a_m}\int_l^{a_m}u(l,a,t)\, \ud a\,\ud l,
\end{equation}
(note the finite maximal age) denote the juvenile and adult populations respectively, but crucially without any further monotonicity assumptions on $\mu$ and $\beta$.

Here we point out (without elaborating the details) that following the lines presented in the previous section, the positive steady state problem for  model \eqref{selmut} could be treated, for example in case of the following choice of $E_1$ and $E_2$: 
\begin{align*}
E_1[u(l,a,t)]:=& \frac{\displaystyle\int_0^{a_m}\left(\alpha\int_0^au(l,x,t)\,\ud x+\int_a^{a_m}u(l,x,t)\,\ud x\right)\,\ud l}{\displaystyle\int_0^{a_m}\int_0^{a_m}w(x)u(l,x,t)\,\ud x\,\ud l},\\ 
E_2[u(l,a,t)]:=& \displaystyle\int_0^{a_m}\int_l^{a_m}u(l,a,t)\, \ud a\,\ud l\left(=Q(t)\right),
\end{align*}
where $w$ is some non-negative weight function and $\alpha\in [0,1]$.

Finally we mention that in the future we intend to explore the possibility to use our framework in the dynamic
setting, that is, we aim to investigate the possibility of using the (re)formulation \eqref{problem} not just to prove the existence of positive steady states, but also to establish further qualitative results, for example results concerning the global (in)stability of the extinction steady state.  We anticipate that lifting the problem into the extra\-po\-lated space will prove to be fruitful again.  We also note that it will be worthwhile to explore the connection between our general approach and the theory of monotone dynamical systems and in particular persistence theory (see e.g. \cite{SmithThieme}), and to compare the advantages/disadvantages of the two approaches.

\section*{Acknowledgments}
The authors were partially supported by the research project MTM2014-52402-C03-2P.


\begin{thebibliography}{99}



\bibitem{AB} (MR1322499)
W. Arendt and C. J. K. Batty,
\newblock Principal eigenvalues and perturbation.
\newblock \emph{Oper. Theory Adv. Appl.},  {\bf 75} (1995), 39--55.

\bibitem{AB2} (MR1072082)
W. Arendt and C. J. K Batty,
\newblock Domination and ergodicity for positive semigroups.
\newblock {\emph Proc. Amer. Math. Soc.}, {\bf 114} (1992), 743--747.

\bibitem{AGG} (MR0839450)
W. Arendt, A. Grabosch, G. Greiner, U. Groh, H. P. Lotz, U. Moustakas, R. Nagel, F. Neubrander and U. Schlotterbeck,
\newblock \emph{One-Parameter Semigroups of Positive Operators},
\newblock Springer-Verlag, Berlin, 1986.

\bibitem{BCR} (MR3717380)
C. Barril C, \`A. Calsina and J. Ripoll,
\newblock On the reproduction number of a gut microbiota model.
\newblock \emph{Bull. Math. Biol.}, {\bf 79} (2017), 2727--2746.

\bibitem{CDF} (MR3582552)
\`A. Calsina, O. Diekmann and J. Z. Farkas, 
\newblock Structured populations with distributed recruitment: from PDE to delay formulation.
\newblock \emph{Math. Methods Appl. Sci.}, {\bf 39} (2016), 5175--5191.

\bibitem{CF} (MR2960842)
\`{A}. Calsina and J. Z. Farkas, 
\newblock Steady states in a structured epidemic model with Wentzell boundary condition.
\newblock {\em J. Evol. Equ.},  {\bf 12} (2012), 495--512.

\bibitem{CF2} (MR3180853)
\`{A}. Calsina and J. Z. Farkas, 
\newblock Positive steady states of evolution equations with finite dimensional nonlinearities.
\newblock \emph{SIAM J. Math. Anal.},  {\bf 46} (2014), 1406--1426.

\bibitem{CF3} (MR3490846)
\`{A}. Calsina and J. Z. Farkas, 
\newblock On a strain-structured epidemic model.
\newblock \emph{ Nonlinear Anal. Real World Appl.}, {\bf 31} (2016), 325--342.

\bibitem{CP} (MR3004970)
\`{A}. Calsina and J. M. Palmada,
\newblock Steady states of a selection-mutation model for an age structured population.
\newblock \emph{J. Math. Anal. Appl.}, {\bf 400} (2013), 386--395. 

\bibitem{Clement} (MR0915552)
Ph. Cl\'{e}ment, H. J. A. M. Heijmans, S. Angenent, C. J. van Duijn and B. de Pagter,
\newblock \emph{One-parameter Semigroups},
\newblock North-Holland Publishing Co., Amsterdam, 1987.

\bibitem{CR1973} (MR0341212)
M. G. Crandall and P. H. Rabinowitz,
Bifurcation, perturbation of simple eigenvalues and linearized stability.
\newblock \emph{Arch. Rational Mech. Anal.}, {\bf 52} (1973), 161--180.

\bibitem{CR} (MR0288640)
M. G. Crandall and P. H. Rabinowitz, 
\newblock Bifurcation from simple eigenvalues.
\newblock \emph{ J. Functional Analysis}, {\bf 8} (1971), 321--340.

\bibitem{Cushing} (MR1293670)
J. M. Cushing,
\newblock The dynamics of hierarchical age-structured populations.
\newblock \emph{J. Math. Biol.}, {\bf 32} (1994), 705--729.

\bibitem{CUS} (MR1636703)
J. M. Cushing,
\newblock \emph{An Introduction to Structured Population Dynamics},
\newblock Society for Industrial and Applied Mathematics (SIAM), Philadelphia, 1998.

\bibitem{CD} (MR3521414)
J. M. Cushing and O. Diekmann,
\newblock The many guises of $R_0$ (a didactic note).
\newblock \emph{J. Theoret. Biol.}, {\bf 404} (2016), 295--302.

\bibitem{PG} (MR0757990)
G. Da Prato and P. Grisvard, 
\newblock Maximal regularity for evolution equations by interpolation and extrapolation.
\newblock \emph{J. Funct. Anal.},  {\bf 58} (1984), 107--124.

\bibitem{D-Sch} (MR0764949)
W. Desch and  W. Schappacher,
\newblock On relatively bounded perturbations of linear $C_0$-semigroups.
\newblock \emph{ Ann. Scuola Norm. Sup. Pisa Cl. Sci. (4)}, {\bf 11} (1984), 327--341.

\bibitem{DGKMT} (MR1860461)
O. Diekmann, M. Gyllenberg, H. Huang, M. Kirkilionis, J. A. J. Metz and H. R. Thieme,
\newblock On the formulation and analysis of general deterministic structured population models. II. Nonlinear theory. 
\newblock \emph{J. Math. Biol.}, {\bf 43} (2001), 157--189. 

\bibitem{DGM} [10.1016/S0040-5809(02)00058-8]
O. Diekmann, M. Gyllenberg M and J. A. J. Metz,
\newblock Steady-state analysis of structured population models.
\newblock \emph{Theoretical Population Biology},  {\bf 63} (2003), 309--338.

\bibitem{NAG} (MR1721989)
K-J Engel and R. Nagel,
\newblock \emph{One-Parameter Semigroups for Linear Evolution Equations},
\newblock Springer-Verlag, New York, 2000.

\bibitem{Farkas2018} (MR3868959)
J. Z. Farkas,
\newblock Net reproduction functions for nonlinear structured population models.
\newblock  \emph{Math. Model. Nat. Phenom.}, {\bf 13} (2018), Art.32, 12pp.

\bibitem{FGH} (MR2663320)
J. Z. Farkas, D. M. Green and P. Hinow,
\newblock Semigroup analysis of structured parasite populations.
\newblock  \emph{ Math. Model. Nat. Phenom.}, {\bf 8} (2010), 94--114.

\bibitem{FH} (MR2285538)
J. Z. Farkas and T. Hagen,
\newblock Stability and regularity results for a size-structured population model.
\newblock \emph{J. Math. Anal. Appl.}, {\bf 328} (2007), 119--136.

\bibitem{FarkasHinow} (MR2959246)
J. Z. Farkas and P. Hinow,
\newblock Steady states in hierarchical structured populations with distributed states at birth.
\newblock \emph{Discrete Contin. Dyn. Syst. Ser. B},  {\bf 17} (2012), 2671--2689.

\bibitem{Gre} (MR0904952)
G. Greiner, 
\newblock Perturbing the boundary conditions of a generator.
\newblock \emph{Houston J. Math.}, {\bf 13} (1987), 213--229.

\bibitem{GM} (MR0354068)
M. E. Gurtin and R. C. MacCamy,
\newblock Non-linear age-dependent population dynamics.
\newblock \emph{Arch. Rational Mech. Anal.}, {\bf 54} (1974), 281-300. 

\bibitem{Iannelli} (MR3700352)
M. Iannelli and F. Milner,
\newblock {\em The Basic Approach to Age-structured Population Dynamics}, 
\newblock Lecture Notes on Mathematical Modelling in the Life Sciences. Springer-Verlag, Dordrecht, 2017.

\bibitem{K} (MR1335452)
T. Kato,
\newblock \emph{Perturbation Theory for Linear Operators},
\newblock Springer-Verlag, Berlin-Heidelberg, 1995.

\bibitem{Magal2018}
P. Magal and S. Ruan,
\newblock {\em Theory and applications of abstract semilinear Cauchy problems},
\newblock Applied Mathematical Sciences. Vol. 201. Springer, Switzerland, 2018.

\bibitem{Magal2009} (MR2559965)
P. Magal and S. Ruan,
\newblock Center manifolds for semilinear equations with non-dense domain and applications to Hopf bifurcation in age structured models. 
\newblock \emph{Mem. Amer. Math. Soc.} {\bf 202} (2009), no. 951.

\bibitem{MD} (MR0860959)
J. A. J. Metz and O. Diekmann,
\newblock \emph{The Dynamics of Physiologically Structured Populations},
\newblock  Springer-Verlag, Berlin, 1986.

\bibitem{Olendorf}
R. Olendorf, F. H. Rodd, D. Punzalan, A. E. Houde, C. Hurt, D. N. Reznick and K. A. Hughes,
\newblock Frequency-dependent survival in natural guppy populations.
\newblock \emph{Nature}, {\bf 441} (2006), 633--636.

\bibitem{Sch} (MR0423039)
H. H. Sch\"{a}fer,
\newblock \emph{Banach lattices and positive operators},
\newblock Springer-Verlag, Berlin, 1974.

\bibitem{SmithThieme} (MR2731633)
H. L. Smith and H. R. Thieme,
\newblock \emph{Dynamical Systems and Population Persistence},
\newblock Graduate Studies in Mathematics, 118. American Mathematical Society, Providence, RI, 2011.


\bibitem{HT2} (MR1485364)
H. R. Thieme,
\newblock Remarks on resolvent positive operators and their perturbation.
\newblock \emph{Discrete Contin. Dynam. Systems}, {\bf 4} (1998), 73--90.

\bibitem{HT} (MR2505085)
H. R. Thieme,
\newblock Spectral bound and reproduction number for infinite-dimensional population structure and time heterogeneity.
\newblock \emph{ SIAM J. Appl. Math.}, {\bf 70} (2009), 188--211.

\bibitem{W} (MR0772205)
G. F. Webb,
\newblock \emph{Theory of Nonlinear Age-Dependent Population Dynamics},
\newblock Marcel Dekker, New York, 1985.





\end{thebibliography}
\end{document}